\documentclass[12pt]{amsart}%
\usepackage{amsmath}
\usepackage{amsfonts}
\usepackage{amssymb}
\usepackage{amstext}
\usepackage{graphicx}%
\providecommand{\U}[1]{\protect \rule{.1in}{.1in}}
\providecommand{\U}[1]{\protect \rule{.1in}{.1in}}
\providecommand{\U}[1]{\protect \rule{.1in}{.1in}}
\newtheorem{theorem}{Theorem}[section]
\newtheorem{lemma}{Lemma}[section]
\newtheorem{proposition}{Proposition}[section]
\newtheorem{corollary}{Corollary}[section]

\newtheorem{definition}{Definition}[section]

\numberwithin{equation}{section}

\theoremstyle{remark}
\newtheorem{remark}{Remark}[section]
\numberwithin{equation}{section}
\begin{document}
\title[CR three circle theorem]{On the three-circle theorem and its applications in Sasakian manifolds}
\author{$^{\ast}$Shu-Cheng Chang}
\address{Department of Mathematics and Taida Institute for Mathematical Sciences
(TIMS), National Taiwan University, Taipei 10617, Taiwan\\
Current address : Yau Mathematical Sciences Center, Tsinghua University,
Beijing, China}
\email{scchang@math.ntu.edu.tw }
\author{$^{\dag}$Yingbo Han}
\address{{School of Mathematics and Statistics, Xinyang Normal University}\\
Xinyang,464000, Henan, China}
\email{{yingbohan@163.com}}
\author{$^{\ast \ast}$Chien Lin}
\address{Department of Mathematics, National Tsing Hua University, Hsinchu 30013, Taiwan}
\email{r97221009@ntu.edu.tw}
\thanks{$^{\ast}$Shu-Cheng Chang and $^{\ast \ast}$Chien Lin are partially supported in
part by the MOST of Taiwan.}
\thanks{$^{\dag}$Yingbo Han is partially supported by an NSFC grant No. 11201400 and
Nanhu Scholars Program for Young Scholars of {Xinyang Normal University}.}

\begin{abstract}
This paper mainly focuses on the CR analogue of the three-circle theorem in a
complete noncompact pseudohermitian manifold of vanishing torsion being odd
dimensional counterpart of K\"{a}hler geometry. In this paper, we show that
the CR three-circle theorem holds if its pseudohermitian sectional curvature
is nonnegative. As an application, we confirm the first CR Yau's
uniformization conjecture and obtain the CR analogue of the sharp dimension
estimate for CR holomorphic functions of polynomial growth and its rigidity
when the pseudohermitian sectional curvature is nonnegative. This is also the
first step toward second and third CR Yau's uniformization conjecture.
Moreover, in the course of the proof of the CR three-circle theorem, we derive
CR sub-Laplacian comparison theorem. Then Liouville theorem holds for positive
pseudoharmonic functions in a complete noncompact pseudohermitian
$(2n+1)$-manifold of vanishing torsion and nonnegative pseudohermitian Ricci curvature.

\end{abstract}
\maketitle

\section{Introduction}

In 1896, J. Hadamard (\cite{ha}) published the so-called classical
three-circle theorem which says that, on the annulus $A$ with inner radius
$r_{1}$ and outer radius $r_{2}$, the logarithm for the modulus of a
holomorphic function on the closure $\overline{A}$ of the annulus is convex
with respect to $\log r$ for $r$ lying between $r_{1}$ and $r_{2}$. Recently,
G. Liu (\cite{liu1}) generalized the three-circle theorem to complete
K\"{a}hler manifolds and characterized a necessary and sufficient condition
for the three-circle theorem. Here a complete K\"{a}hler manifold $M$ is said
to satisfy the three-circle theorem if, for any point $p\in M$, $R>0$, any
holomorphic function $f$ on the geodesic ball $B\left(  p,R\right)  $, $\log
M_{f}\left(  r\right)  $ is convex with respect to $\log r$, namely, for
$0<r_{1}\leq r_{2}\leq r_{3}<R$,
\begin{equation}
\log \left(  \frac{r_{3}}{r_{1}}\right)  \log M_{f}\left(  r_{2}\right)
\leq \log \left(  \frac{r_{3}}{r_{2}}\right)  \log M_{f}\left(  r_{1}\right)
+\log \left(  \frac{r_{2}}{r_{1}}\right)  \log M_{f}\left(  r_{3}\right)
\label{101}%
\end{equation}
where $M_{f}\left(  r\right)  =\underset{x\in B\left(  p,r\right)  }{\sup
}\left \vert f\left(  x\right)  \right \vert $. More precisely, he showed that a
complete K\"{a}hler manifold satisfies the three-circle theorem if and only if
its holomorphic sectional curvature is nonnegative. The proof employed the
Hessian comparison and the maximum principle. There are many substantial
applications pertaining to the three-circle theorem, especially to the
uniformization-type problems proposed by Yau (\cite{scy}) on complete
K\"{a}hler manifolds with nonnegative holomorphic bisectional curvature. It
could be summarized as follows. The first Yau's uniformization conjecture is
that if $M$ is a complete noncompact $n$-dimensional K\"{a}hler manifold with
nonnegative holomorphic bisectional curvature, then
\[
\dim_{%
\mathbb{C}
}\left(  \mathcal{O}_{d}\left(  M^{n}\right)  \right)  \leq \dim_{%
\mathbb{C}
}\left(  \mathcal{O}_{d}\left(
\mathbb{C}
^{n}\right)  \right)  .
\]
The equality holds if and only if $M$ is isometrically biholomorphic to $%
\mathbb{C}
^{n}$. Here $\mathcal{O}_{d}\left(  M^{n}\right)  $ denotes the family of all
holomorphic functions on a complete $n$-dimensional K\"{a}hler manifold $M$ of
polynomial growth of degree at most $d$. In \cite{n}, Ni established the
validity of this conjecture by deriving the monotonicity formula for the heat
equation under the assumption that $M$ has maximal volume growth
\[
\lim_{r\rightarrow+\infty}\frac{Vol\left(  B_{p}\left(  r\right)  \right)
}{r^{2n}}\geq c
\]
for a fixed point $p$ and a positive constant $c$. Later, in \cite{cfyz}, the
authors improved Ni's result without the assumption of maximal volume growth.
In recent years, G. Liu (\cite{liu1}) generalized the sharp dimension estimate
by only assuming that $M$ admits nonnegative holomorphic sectional curvature.
Note that there are noncompact complex manifolds admitting complete K\"{a}hler
metrics with positive holomorphic sectional curvature but not admitting
complete K\"{a}hler metrics with nonnegative Ricci curvature (see \cite{hi}).

The second Yau's uniformization conjecture is that if $M$ is a complete
noncompact $n$-dimensional K\"{a}hler manifold with nonnegative holomorphic
bisectional curvature, then the ring $\mathcal{O}_{P}\left(  M\right)  $ of
all holomorphic functions of polynomial growth is finitely generated. This one
was solved completely by G. Liu (\cite{liu2}) quite recently. He mainly
deployed four techniques to attack this conjecture via Cheeger-Colding-Tian's
theory (\cite{chco1}, \cite{chco2}, \cite{cct}), methods of heat flow
developed by Ni and Tam (\cite{n}, \cite{nt1}, \cite{nt4}), H\"{o}rmander
$L^{2}$-estimate of $\overline{\partial}$ (\cite{de}) and three-circle theorem
(\cite{liu1}) as well.

The third Yau's uniformization conjecture is that if $M$ is a complete
noncompact $n$-dimensional K\"{a}hler manifold with positive holomorphic
bisectional curvature, then $M$ is biholomorphic to the standard
$n$-dimensional complex Euclidean space $%
\mathbb{C}
^{n}$. The first giant progress relating to the third conjecture could be
attributed to Mok, Siu and Yau. In their papers (\cite{msy} and \cite{m1}),
they showed that, under the assumptions of the maximal volume growth and the
scalar curvature $R\left(  x\right)  $ decays as
\[
0\leq R\left(  x\right)  \leq \frac{C}{\left(  1+d\left(  x,x_{0}\right)
\right)  ^{2+\epsilon}}%
\]
for some positive constant $C$ and any arbitrarily small positive number
$\epsilon$, a complete noncompact $n$-dimensional K\"{a}hler manifold $M$ with
nonnegative holomorphic bisectional curvature is isometrically biholomorphic
to $%
\mathbb{C}
^{n}$. A Riemannian version was solved in \cite{gw2} shortly afterwards. Since
then there are several further works aiming to prove the optimal result and
reader is referred to \cite{m2}, \cite{ctz}, \cite{cz}, \cite{n2}, \cite{nt1},
\cite{nt2} and \cite{nst}. For example, A. Chau and L. F. Tam (\cite{ct})
proved that a complete noncompact K\"{a}hler manifold with bounded nonnegative
holomorphic bisectional curvature and maximal volume growth is biholomorphic
to $%
\mathbb{C}
^{n}$. Recently, G. Liu (\cite{liu3}) confirmed Yau's uniformization
conjecture when $M$ has maximal volume growth. Later, M.-C. Lee and L.-F. Tam
(\cite{lt}) also confirmed Yau's uniformization conjecture with the maximal
volume growth condition.

For the corresponding first uniformization conjectures in a complete
noncompact pseudohermitian $(2n+1)$-manifold of vanishing torsion (i.e.
Sasakian manifold) which is an odd dimensional counterpart of K\"{a}hler
geometry (see the next section for its definition and some properties), it was
settled that the CR sharp dimension estimate for CR holomorphic functions of
polynomial growth with nonnegative pseudohermitian bisectional curvature in
\cite{chl} of which proof is inspired primarily from \cite{n} and \cite{cfyz}.
So it's natural to concern whether the second and third CR Yau's
uniformization conjectures hold as well. However, as inspired by recent works
of G. Liu (\cite{liu2}, \cite{liu3}), it is crucial to have the CR analogue of
the three-circle theorem which is a step towards establishing the validity of
such CR Yau's uniformization conjectures.

In this paper, we mainly focus on the CR three-circle theorem in a complete
noncompact pseudohermitian $\left(  2n+1\right)  $-manifold of vanishing
torsion with nonnegative pseudohermitian sectional curvature which is weak
than nonnegative pseudohermitian bisectional curvature.

A smooth complex-valued function on a pseudohermitian $\left(  2n+1\right)
$-manifold $\left(  M,J,\theta \right)  $ is called CR-holomorphic if%
\[
\overline{\partial}_{b}f=0.
\]

We recall $\mathcal{O}^{CR}(M)$ the family of all CR-holomorphic functions $f$
with $Tf(x)=f_{0}(x)=0$ (\cite{chl})
\[
\mathcal{O}^{CR}(M)=\{f(x)\in C_{%
\mathbb{C}
}^{\infty}\left(  M\right)  \ |\overline{\partial}_{b}f(x)=0\text{
}\mathrm{and\ }\ f_{0}(x)=0\text{ }\  \},
\]
where the extra condition $Tf(x)=0$ is included, the interested readers could
refer to \cite{chl} or \cite{fow}.

Next we give the definition of the CR three-circle theorem generalizing the
classical Hadamard's three-circle theorem to CR manifolds :

\begin{definition}
Let $\left(  M,J,\theta \right)  $ be a complete pseudohermitian $\left(
2n+1\right)  $-manifold. $\left(  M,J,\theta \right)  $ is said to satisfy the
CR three-circle theorem if, for any point $p\in M$, any positive number $R>0$,
and any function $f\in \mathcal{O}^{CR}(B_{cc}\left(  p,R\right)  )$ on the
ball $B_{cc}\left(  p,R\right)  $, $\log M_{f}\left(  r\right)  $ is convex
with respect to $\log r$ for $0<r<R$. Here $M_{f}\left(  r\right)
=\underset{x\in B_{cc}\left(  p,r\right)  }{\sup}\left \vert f\left(  x\right)
\right \vert $ and $B_{cc}\left(  p,R\right)  $ is the Carnot-Carath\'{e}odory
ball centered at $p$ with radius $R$.
\end{definition}

Now we state our main theorem in this paper as follows:

\begin{theorem}
\label{m302}If $\left(  M,J,\theta \right)  $ is a complete noncompact
pseudohermitian $\left(  2n+1\right)  $-manifold with vanishing torsion, then
the CR three-circle theorem holds on $M$ if the pseudohermitian sectional
curvature is nonnegative; moreover, we have that, for any $f\in \mathcal{O}%
^{CR}(M)$,
\begin{equation}
\frac{M_{f}\left(  kr\right)  }{M_{f}\left(  r\right)  } \label{314}%
\end{equation}
is increasing with respect to $r$ for any positive number $k\geq1$.
\end{theorem}

\begin{remark}
1. In the course of the proof of the CR three-circle theorem, we derive the
following CR sub-Laplacian comparison%
\[
\Delta_{b}r\leq \frac{(2n-1)}{r}%
\]
if $\left(  M,J,\theta \right)  $ is a complete noncompact pseudohermitian
$\left(  2n+1\right)  $-manifold of vanishing torsion with nonnegative
pseudohermitian bisectional curvature. It is sharp in case of $n=1.$ See
Corollary \ref{c31} for details.

2. As a consequence of the sub-Laplacian comparison, Liouville theorem holds
for positive pseduoharmonic functions in a complete noncompact pseudohermitian
$(2n+1)$-manifold of vanishing torsion and nonnegative pseudohermitian Ricci
curvature (\cite{cklt}).
\end{remark}

As an application of the preceding theorem, we have the enhanced version of
the sharp dimension estimate for CR holomorphic functions of polynomial growth
and its rigidity which is served a generalization of authors previous results
(\cite{chl}).

\begin{theorem}
\label{m308}If $\left(  M,J,\theta \right)  $ is a complete noncompact
pseudohermitian $\left(  2n+1\right)  $-manifold of vanishing torsion with
nonnegative pseudohermitian sectional curvature, then%
\[
\dim_{%
\mathbb{C}
}\left(  \mathcal{O}_{d}^{CR}(M)\right)  \leq \dim_{%
\mathbb{C}
}\left(  \mathcal{O}_{d}^{CR}(\mathbb{H}^{n})\right)
\]
for any positive integer $d\in%
\mathbb{N}
$; morever, if $M$ is simply connected, then the equality holds if only if
$\left(  M,J,\theta \right)  $ is CR-isomorphic to $\left(  2n+1\right)
$-dimensional Heisenberg group $\mathbb{H}^{n}$.
\end{theorem}

\begin{remark}
1. In the forthcoming paper, one of the most important application for CR
three-circle theorem, we expect that there exists a nonconstant CR-holomorphic
function of polynomial growth on a complete noncompact pseudohermitian
$\left(  2n+1\right)  $-manifold $\left(  M,J,\theta \right)  $ of vanishing
torsion, nonnegative pseudohermitian bisectional curvature and maximal volume
growth. This is the first step toward the second CR Yau's uniformization
conjecture (\cite{liu2}).

2. As in the paper of G. Liu (\cite{liu3}), by applying Cheeger-Colding theory
for the Webster metric on CR manifolds and H\"{o}rmander $L^{2}$-technique of
$\overline{\partial}_{b}$ on the space of basic forms and this CR three-circle
theorem, we shall work on the third CR Yau's uniformization conjecture as well.
\end{remark}

Besides, the CR three-circle theorem could\ be extended to the case when $M$
admits the pseudohermitian sectional curvature bounded below:

\begin{theorem}
\label{m303}Let $\left(  M,J,\theta \right)  $ be a complete noncompact
pseudohermitian $\left(  2n+1\right)  $-manifold, $r\left(  x\right)
=d_{cc}\left(  p,x\right)  $ and $Z_{1}=\frac{1}{\sqrt{2}}\left(  \nabla
_{b}r-iJ\nabla_{b}r\right)  $ for some fixed point $p\in M$. If the
pseudohermitian sectional curvature $R_{1\overline{1}1\overline{1}}\left(
x\right)  $ has the inequality
\begin{equation}
R_{1\overline{1}1\overline{1}}\left(  x\right)  \geq g\left(  r\left(
x\right)  \right)  \label{305}%
\end{equation}
for $g\in C^{0}\left(  \left[  0,+\infty \right)  \right)  $ and the
pseudohermitian torsion vanishes, a function $u\left(  r\right)  \in
C^{1}\left(
\mathbb{R}
^{+}\right)  $ satisfies
\begin{equation}
2u^{2}+u^{\prime}+\frac{g}{2}\geq0 \label{306}%
\end{equation}
with $u\left(  r\right)  \thicksim \frac{1}{2r}$ as $r\rightarrow0^{+}$ and a
function $h\left(  r\right)  \in C^{1}\left(
\mathbb{R}
^{+}\right)  $ satisfies
\begin{equation}
h^{\prime}\left(  r\right)  >0 \label{307}%
\end{equation}
and%
\begin{equation}
\frac{1}{2}h^{\prime \prime}\left(  r\right)  +h^{\prime}\left(  r\right)
u\left(  r\right)  \leq0 \label{308}%
\end{equation}
with $h\left(  r\right)  \thicksim \log r$ as $r\rightarrow0^{+}$, then $\log
M_{f}\left(  r\right)  $ is convex with respect to the function $h\left(
r\right)  $ for $f\in \mathcal{O}^{CR}(M)$; moreover, if the vanishing order
$ord_{p}\left(  f\right)  $ of $f\in \mathcal{O}^{CR}(M)$ at $p$ is equal to
$d$, then $\frac{M_{f}\left(  r\right)  }{\exp \left(  dh\left(  r\right)
\right)  }$ is increasing with respect to $r$.
\end{theorem}

As precedes, there is also a sharp dimension estimate when the pseudohermitian
sectional curvature is asymptotically nonnegative.

\begin{theorem}
\label{m309}If $\left(  M,J,\theta \right)  $ is a complete noncompact
pseudohermitian $\left(  2n+1\right)  $-manifold, $r\left(  x\right)
=d_{cc}\left(  p,x\right)  $ for some fixed point $p\in M$, and there are two
positive constants $\epsilon$, $A$ such that
\[
R_{j\overline{j}j\overline{j}}\left(  x\right)  \geq-\frac{A}{\left(
1+r\left(  x\right)  \right)  ^{2+\epsilon}}%
\]
for any $Z_{j}\in T_{x}^{1,0}M$ with $\left \vert Z_{j}\right \vert =1$, then
there is a constant $C\left(  \epsilon,A\right)  >0$ such that, for any $d\in%
\mathbb{N}
$,
\[
\dim_{%
\mathbb{C}
}\left(  \mathcal{O}_{d}^{CR}(M)\right)  \leq C\left(  \epsilon,A\right)
d^{n}.
\]
Furthermore, if $d\leq e^{-\frac{3A}{\epsilon}}$, then
\[
\dim_{%
\mathbb{C}
}\left(  \mathcal{O}_{d}^{CR}(M)\right)  =1.
\]
Finally, if
\[
\frac{A}{\epsilon}\leq \frac{1}{4d}%
\]
for $d\in%
\mathbb{N}
$, then we have%
\[
\dim_{%
\mathbb{C}
}\left(  \mathcal{O}_{d}^{CR}(M)\right)  \leq \dim_{%
\mathbb{C}
}\left(  \mathcal{O}_{d}^{CR}(\mathbb{H}^{n})\right)  .
\]

\end{theorem}

The method we adopt here is inspired from \cite{liu1}. This paper is organized
as follows. In Section $2$, we introduce some basic notions about
pseudohermitian manifolds and the necessary results for this paper. In Section
$3$, we show that the CR analogue of the three-circle theorem and some of its
applications; specially, we confirm the first CR Yau's uniformization
conjecture on the sharp dimension estimate for CR holomorphic functions of
polynomial growth and its rigidity. As a by-product, we obtain the CR
sub-Laplacian comparison theorem. In Section $4$, we generalize the CR
three-circle theorem to the case when the pseudohermitian sectional curvature
is bounded below. It enables us to derive the dimension estimate when the
pseudohermitian sectional curvature is asymptotically nonnegative.

\section{Preliminaries}

We introduce some basic materials about a pseudohermitian manifold (see
\cite{l} and \cite{dt} for more details). Let $(M,\xi)=\left(  M,J,\theta
\right)  $ be a $(2n+1)$-dimensional, orientable, contact manifold with the
contact structure $\xi$. A CR structure compatible with $\xi$ is an
endomorphism $J:\xi \rightarrow \xi$ such that $J^{2}=-Id$. We also assume that
$J$ satisfies the integrability condition: If $X$ and $Y$ are in $\xi$, then
so are $[JX,Y]+[X,JY]$ and $J([JX,Y]+[X,JY])=[JX,JY]-[X,Y]$. A contact
manifold $\left(  M,\xi \right)  =\left(  M,J,\theta \right)  $ with a CR
structure $J$ compatible with $\xi$ together with a contact form $\theta$ is
called a pseudohermitian manifold or a strictly pseudoconvex CR manifold as
well. Such a choice induces a unique vector field $T\in \Gamma \left(
TM\right)  $ transverse to the contact structure $\xi$, which is called the
Reeb vector field or the characteristic vector field of $\theta$ such that
$\iota_{T}\theta=0=\iota_{T}d\theta$. A CR structure $J$ could be extended to
the complexified space $\xi^{%
\mathbb{C}
}=%
\mathbb{C}
\otimes \xi$ of the contact structure $\xi$and decompose it into the direct sum
of $T_{1,0}M$ and $T_{0,1}M$ which are eigenspaces of $J$ corresponding to the
eigenvalues $1$ and $-1$, respectively.

Let $\left \{  T,Z_{\alpha},Z_{\bar{\alpha}}\right \}  _{\alpha \in I_{n}}$ be a
frame of $TM\otimes \mathbb{C}$, where $Z_{\alpha}$ is any local frame of
$T_{1,0}M,\ Z_{\bar{\alpha}}=\overline{Z_{\alpha}}\in T_{0,1}M$, and $T$ is
the Reeb vector field (or the characteristic direction) and $I_{n}=\left \{
1,2,...,n\right \}  $. Then $\left \{  \theta^{\alpha},\theta^{\bar{\alpha}%
},\theta \right \}  $, the coframe dual to $\left \{  Z_{\alpha},Z_{\bar{\alpha}%
},T\right \}  $, satisfies
\[
d\theta=ih_{\alpha \overline{\beta}}\theta^{\alpha}\wedge \theta^{\overline
{\beta}}%
\]
for some positive definite matrix of functions $(h_{\alpha \bar{\beta}})$.
Usually, we assume such matrix $(h_{\alpha \bar{\beta}})$ is the identity
matrix A pseudohermitian manifold $\left(  M,J,\theta \right)  $ is called a
Sasakian manifold if the pseudohermitian torsion $\tau=\iota_{T}T_{D}=0$

The Levi form $\left \langle \ ,\  \right \rangle _{L_{\theta}}$ is the Hermitian
form on $T_{1,0}M$ defined by%
\[
\left \langle Z,W\right \rangle _{L_{\theta}}=-i\left \langle d\theta
,Z\wedge \overline{W}\right \rangle .
\]

We can extend $\left \langle \ ,\  \right \rangle _{L_{\theta}}$ to $T_{0,1}M$ by
defining $\left \langle \overline{Z},\overline{W}\right \rangle _{L_{\theta}%
}=\overline{\left \langle Z,W\right \rangle }_{L_{\theta}}$ for all $Z,W\in
T_{1,0}M$. The Levi form induces naturally a Hermitian form on the dual bundle
of $T_{1,0}M$, denoted by $\left \langle \ ,\  \right \rangle _{L_{\theta}^{\ast
}}$, and hence on all the induced tensor bundles. Integrating the Hermitian
form (when acting on sections) over $M$ with respect to the volume form
$d\mu=\theta \wedge(d\theta)^{n}$, we get an inner product on the space of
sections of each tensor bundle.

The Tanaka-Webster connection of $(M,J,\theta)$ is the connection $\nabla$ on
$TM\otimes \mathbb{C}$ (and extended to tensors) given, in terms of a local
frame $Z_{\alpha}\in T_{1,0}M$, by%

\[
\nabla Z_{\alpha}=\omega_{\alpha}{}^{\beta}\otimes Z_{\beta},\quad \nabla
Z_{\bar{\alpha}}=\omega_{\bar{\alpha}}{}^{\bar{\beta}}\otimes Z_{\bar{\beta}%
},\quad \nabla T=0,
\]
where $\omega_{\alpha}{}^{\beta}$ are the $1$-forms uniquely determined by the
following equations:%

\[
\left \{
\begin{array}
[c]{l}%
d\theta^{\alpha}+\omega_{\beta}^{\alpha}\wedge \theta^{\beta}=\theta \wedge
\tau^{\alpha}\\
\tau_{\alpha}\wedge \theta^{\alpha}=0\\
\omega_{\beta}^{\alpha}+\omega_{\overline{\alpha}}^{\overline{\beta}}=0
\end{array}
\right.  .
\]

We can write (by Cartan lemma) $\tau_{\alpha}=A_{\alpha \gamma}\theta^{\gamma}$
with $A_{\alpha \gamma}=A_{\gamma \alpha}$. The curvature of Tanaka-Webster
connection $\nabla$, expressed in terms of the coframe $\{ \theta=\theta
^{0},\theta^{\alpha},\theta^{\bar{\alpha}}\}$, is%
\[
\left \{
\begin{array}
[c]{c}%
\Pi_{\beta}^{\alpha}=\overline{\Pi_{\overline{\beta}}^{\overline{\alpha}}%
}=d\omega_{\beta}^{\alpha}+\omega_{\gamma}^{\alpha}\wedge \omega_{\beta
}^{\gamma}\\
\Pi_{0}^{\alpha}=\Pi_{\alpha}^{0}=\Pi_{0}^{\overline{\alpha}}=\Pi
_{\overline{\alpha}}^{0}=\Pi_{0}^{0}=0
\end{array}
\right.  .
\]

Webster showed that $\Pi_{\beta}{}^{\alpha}$ can be written%
\[
\Omega_{\beta}^{\alpha}=\Pi_{\beta}^{\alpha}+i\tau^{\alpha}\wedge \theta
_{\beta}-i\theta^{\alpha}\wedge \tau_{\beta}=R_{\beta \text{ \ }\gamma
\overline{\delta}}^{\text{ \ }\alpha}\theta^{\gamma}\wedge \theta
^{\overline{\delta}}+W_{\beta \gamma}^{\alpha}\theta^{\gamma}\wedge
\theta-W_{\beta \overline{\gamma}}^{\alpha}\theta^{\overline{\gamma}}%
\wedge \theta
\]
where the coefficients satisfy
\[
R_{\beta \bar{\alpha}\rho \bar{\sigma}}=\overline{R_{\alpha \bar{\beta}\sigma
\bar{\rho}}}=R_{\bar{\alpha}\beta \bar{\sigma}\rho}=R_{\rho \bar{\alpha}%
\beta \bar{\sigma}};\text{ }W_{\beta \gamma}^{\alpha}=A_{\beta \gamma},^{\alpha
};\text{ }W_{\beta \overline{\gamma}}^{\alpha}=A_{\overline{\gamma}}{}^{\alpha
},_{\beta}.
\]

Besides $R_{\alpha \overline{\beta}\gamma \overline{\delta}}$, the other part of
the curvature of Tanaka-Webster connection are clear:%
\begin{equation}
\left \{
\begin{array}
[c]{l}%
R_{\alpha \overline{\beta}\gamma \overline{\mu}}=-2i\left(  A_{\alpha \mu}%
\delta_{\beta \gamma}-A_{\alpha \gamma}\delta_{\beta \mu}\right) \\
R_{\alpha \overline{\beta}\overline{\gamma}\overline{\mu}}=-2i\left(
A_{\overline{\beta}\overline{\mu}}\delta_{\alpha \gamma}-A_{\overline{\beta
}\overline{\gamma}}\delta_{\alpha \mu}\right) \\
R_{\alpha \overline{\beta}0\gamma}=A_{\gamma \alpha,\overline{\beta}}\\
R_{\alpha \overline{\beta}0\overline{\gamma}}=-A_{\overline{\beta}%
\overline{\gamma},\alpha}%
\end{array}
\right.  . \label{324}%
\end{equation}

Here $R_{\beta \text{ \ }\gamma \overline{\delta}}^{\text{ \ }\alpha}$ is the
pseudohermitian curvature tensor field, $R_{\alpha \bar{\beta}}=R_{\gamma}%
{}^{\gamma}{}_{\alpha \bar{\beta}}$ is the pseudohermitian Ricci curvature
tensor field and $A_{\alpha \beta}$\ is the pseudohermitian torsion.
$R=h^{\alpha \overline{\beta}}R_{\alpha \bar{\beta}}$ denotes the
pseudohermitian scalar curvature. Moreover, we define the pseudohermitian
bisectional curvature tensor field%
\[
R_{\alpha \bar{\alpha}\beta \overline{\beta}}(X,Y)=R_{\alpha \bar{\alpha}%
\beta \overline{\beta}}X_{\alpha}X_{\overline{\alpha}}Y_{\beta}Y_{\bar{\beta}%
},
\]
the bitorsion tensor field
\[
T_{\alpha \overline{\beta}}(X,Y)=\frac{1}{i}(A_{\alpha \gamma}X^{\gamma
}Y_{\overline{\beta}}-A_{\overline{\beta}\overline{\gamma}}X^{\overline
{\gamma}}Y_{\alpha}),
\]
and the torsion tensor field%
\[
Tor\left(  X,Y\right)  =tr\left(  T_{\alpha \overline{\beta}}\right)  =\frac
{1}{i}(A_{\alpha \beta}X^{\beta}Y^{\alpha}-A_{\overline{\alpha}\overline{\beta
}}X^{\overline{\beta}}Y^{\overline{\alpha}}),
\]
where $X=X^{\alpha}Z_{\alpha},Y=Y^{\alpha}Z_{\alpha}$ in $T_{1,0}M$.

We will denote the components of the covariant derivatives with indices
preceded by comma. The indices $\{0,\alpha,\bar{\alpha}\}$ indicate the
covariant derivatives with respect to $\{T,Z_{\alpha},Z_{\bar{\alpha}}\}$. For
the covariant derivatives of a real-valued function, we will often omit the
comma, for instance, $u_{\alpha}=Z_{\alpha}u,\ u_{\alpha \bar{\beta}}%
=Z_{\bar{\beta}}Z_{\alpha}u-\omega_{\alpha}{}^{\gamma}(Z_{\bar{\beta}%
})Z_{\gamma}u$. The subgradient $\nabla_{b}\varphi$ of a smooth real-valued
function $\varphi$ is defined by
\[
\left \langle \nabla_{b}\varphi,Z\right \rangle _{L_{\theta}}=Z\varphi
\]
for $Z\in \Gamma \left(  \xi \right)  $ where $\Gamma \left(  \xi \right)  $
denotes the family of all smooth vector fields tangent to the conact plane
$\xi$. We could locally write the subgradient $\nabla_{b}\varphi$ as
\[
\nabla_{b}u=u^{\alpha}Z_{\alpha}+u^{\overline{\alpha}}Z_{\bar{\alpha}}.
\]

Accordingly, we could define the subhessian $Hess_{b}$ as the complex linear
map%
\[
Hess_{b}:T_{1,0}M\oplus T_{0,1}M\longrightarrow T_{1,0}M\oplus T_{0,1}M
\]
by%
\[
\left(  Hess_{b}\varphi \right)  Z=\nabla_{Z}\nabla_{b}\varphi
\]
for $Z\in \Gamma \left(  \xi \right)  $ and a smooth real-valued function
$\varphi$.

Also, the sub-Laplacian is defined by%
\[
\Delta_{b}u=tr\left(  Hess_{b}u\right)  =u_{\alpha}{}^{\alpha}+u^{\alpha}%
{}_{\alpha}\text{.}%
\]

Now we recall the following commutation relations (see \cite{l}). Let
$\varphi$ be a smooth real-valued function, $\sigma=\sigma_{\alpha}%
\theta^{\alpha}$ be a $\left(  1,0\right)  $-form and $\varphi_{0}=T\varphi,$
then we have%
\begin{equation}
\left \{
\begin{array}
[c]{cll}%
\varphi_{\alpha \beta} & = & \varphi_{\beta \alpha},\\
\varphi_{\alpha \bar{\beta}}-\varphi_{\bar{\beta}\alpha} & = & ih_{\alpha
\overline{\beta}}\varphi_{0}\\
\varphi_{0\alpha}-\varphi_{\alpha0} & = & A_{\alpha \beta}\varphi^{\beta}\\
\sigma_{\alpha,\beta \gamma}-\sigma_{\alpha,\gamma \beta} & = & i\left(
A_{\alpha \gamma}\sigma_{\beta}-A_{\alpha \beta}\sigma_{\gamma}\right) \\
\sigma_{\alpha,\overline{\beta}\overline{\gamma}}-\sigma_{\alpha
,\overline{\gamma}\overline{\beta}} & = & -i(h_{\alpha \overline{\gamma}%
}A_{\overline{\beta}}^{\delta}\sigma_{\delta}-h_{\alpha \overline{\beta}%
}A_{\overline{\gamma}}^{\delta}\sigma_{\delta})\\
\sigma_{\alpha,\beta \overline{\gamma}}-\sigma_{\alpha,\overline{\gamma}\beta}
& = & ih_{\beta \overline{\gamma}}\sigma_{\alpha,}{}_{0}+R_{\alpha}^{\text{
}\delta}{}_{\beta \overline{\gamma}}\sigma_{\delta}\\
\sigma_{\alpha,0\beta}-\sigma_{\alpha,\beta0} & = & A^{\overline{\gamma}}%
{}_{\beta}\sigma_{\alpha,\bar{\gamma}}-A_{\alpha \beta,\bar{\gamma}}%
\sigma^{\overline{\gamma}}\\
\sigma_{\alpha,0\bar{\beta}}-\sigma_{\alpha,\bar{\beta}0} & = & A^{\gamma}%
{}_{\bar{\beta}}\sigma_{\alpha,\gamma}+A_{\bar{\gamma}\bar{\beta},\alpha
}\sigma^{\overline{\gamma}}%
\end{array}
\right.  . \label{201}%
\end{equation}

Subsequently, we introduce the notion about the Carnot-Carath\'{e}odory distance.

\begin{definition}
A piecewise smooth curve $\gamma:[0,1]\rightarrow \left(  M,\xi \right)  $ is
said to be horizontal if $\gamma^{\prime}\left(  t\right)  \in \xi$ whenever
$\gamma^{\prime}\left(  t\right)  $ exists. The length of $\gamma$ is then
defined by
\[
L(\gamma)=\int_{0}^{1}\left \langle \gamma^{\prime}\left(  t\right)
,\gamma^{\prime}\left(  t\right)  \right \rangle _{L_{\theta}}^{\frac{1}{2}%
}dt.
\]
The Carnot-Carath\'{e}odory distance between two points $p,\ q\in M$ is
\[
d_{cc}(p,q)=\inf \left \{  L(\gamma)|\  \gamma \in C_{p,q}\right \}  ,
\]
where $C_{p,q}$ is the set of all horizontal curves joining $p$ and $q$. We
say $\left(  M,\xi \right)  $ is complete if it's complete as a metric space.
By Chow's connectivity theorem, there always exists a horizontal curve joining
$p$ and $q$, so the distance is finite. The diameter $d_{c}$ is defined by
\[
d_{c}(M)=\sup \left \{  d_{cc}(p,q)|\ p,q\in M\right \}  .
\]
Note that there is a minimizing geodesic joining $p$ and $q$ so that its
length is equal to the distance $d_{cc}(p,q).$
\end{definition}

For any fixed point $x\in M$, a CR-holomorphic function $f$ is called to be of
polynomial growth if there are a nonnegative number $d$ and a positive
constant $C=C\left(  x,d,f\right)  $, depending on $x$, $d$ and $f$, such that%
\[
\left \vert f\left(  y\right)  \right \vert \leq C\left(  1+d_{cc}\left(
x,y\right)  \right)  ^{d}%
\]
for all $y\in M$, where $d_{cc}\left(  x,y\right)  $ denotes the
Carnot-Carath\'{e}odory distance between $x$ and $y.$ Furthermore, we could
define the degree of a CR-holomorphic function $f$ of polynomial growth by%
\begin{equation}
\deg \left(  f\right)  =\inf \left \{  d\geq0\left \vert
\begin{array}
[c]{c}%
\left \vert f\left(  y\right)  \right \vert \leq C\left(  1+d_{cc}\left(
x,y\right)  \right)  ^{d}\text{ \ }\forall \text{ }y\in M,\\
\mathrm{for}\text{ }\mathrm{some\ }\text{ }d\geq0\text{ \ }\mathrm{and\ }%
\text{ }C=C\left(  x,d,f\right)
\end{array}
\right.  \right \}  . \label{203}%
\end{equation}

With these notions, we could define the family $\mathcal{O}_{d}^{CR}(M)$ of
all CR-holomorphic functions $f$ of polynomial growth of degree at most $d$
with $Tf(x)=f_{0}(x)=0:$
\begin{equation}
\mathcal{O}_{d}^{CR}(M)=\{f\in \mathcal{O}^{CR}(M)\ |\text{ }\  \deg \left(
f\right)  \leq d\text{ }\}. \label{202}%
\end{equation}

Finally, we denote by $ord_{p}\left(  f\right)  =\max \left \{  m\in%
\mathbb{N}
\text{ }|\ D^{\alpha}f\left(  p\right)  =0\text{,\ }\forall \text{ }\left \vert
\alpha \right \vert <m\  \right \}  $ the vanishing order of CR-holomorphic
function $f$ at $p$ where $D^{\alpha}=%
{\displaystyle \prod \limits_{j\in I_{n}}}
Z_{j}^{\alpha_{j}}$ with $\alpha=\left(  \alpha_{1},\alpha_{2,}...,\alpha
_{n}\right)  $.

\section{CR\ Three-Circle Theorem}

In this section, we will derive the CR analogue of three-circle theorem on a
complete noncompact pseudohermitian $(2n+1)$-manifold. Before that, we need a
lemma which is essential in the course of the proof of the CR three-circle
theorem as follows:

\begin{lemma}
\label{m301}If $\left(  M,J,\theta \right)  $ is a complete noncompact
pseudohermitian $\left(  2n+1\right)  $-manifold of vanishing torsion with
nonnegative pseudohermitian sectional curvature, the Carnot-Carath\'{e}odory
distance function $r\left(  x\right)  =d_{cc}\left(  p,x\right)  $ from a
fixed point $p$ to a point $x$ in $M$ is smooth at $q\in M$ and $Z_{1}%
=\frac{1}{\sqrt{2}}\left(  \nabla_{b}r-iJ\nabla_{b}r\right)  $, then
\begin{equation}
r_{1\overline{1}}\leq \frac{1}{2r}. \label{301}%
\end{equation}
In particular, we have
\[
\left(  \log r\right)  _{1\overline{1}}\leq0.
\]

\end{lemma}

\begin{proof}
Let $\left \{  e_{j},e_{\widetilde{j}},T\right \}  _{j\in I_{n}}$be an
orthonormal frame at $q$ where $e_{\widetilde{j}}=Je_{j}$ and $e_{1}%
=\nabla_{b}r$. By Corollary 2.3 in \cite{dz} and vanishing pseudohermitian
torsion, we could parallel transport such frame at $q$ to obtain the
orthonormal frame along the radial $\nabla$-geodesic $\gamma$ from $p$ to $q$.
Hence we have an orthonormal frame $\left \{  Z_{j},Z\overline{_{j}},T\right \}
_{j\in I_{n}}$ along $\gamma$ where $Z_{j}=\frac{1}{\sqrt{2}}\left(
e_{j}-ie_{\widetilde{j}}\right)  $ and $Z_{\overline{j}}=\overline{Z_{j}}$. By
the fact that $\gamma$ is the $\nabla$-geodesic, we have
\[%
\begin{array}
[c]{ccl}%
r_{11} & = & -\frac{1}{2}\left(  ie_{2}e_{1}+e_{2}e_{2}\right)  r-\left(
\nabla_{Z_{1}}Z_{1}\right)  r\\
& = & -\frac{1}{2}\left(  ie_{2}e_{1}+e_{2}e_{2}\right)  r+\frac{1}{2}\left[
i\nabla_{\left(  J\nabla_{b}r\right)  }\nabla_{b}r+J\left(  \nabla_{\left(
J\nabla_{b}r\right)  }\nabla_{b}r\right)  \right]
\end{array}
\]
and%
\[%
\begin{array}
[c]{ccl}%
r_{1\overline{1}} & = & \frac{1}{2}\left(  ie_{2}e_{1}+e_{2}e_{2}\right)
r-\left(  \nabla_{Z_{\overline{1}}}Z_{1}\right)  r\\
& = & \frac{1}{2}\left(  ie_{2}e_{1}+e_{2}e_{2}\right)  r-\frac{1}{2}\left[
i\nabla_{\left(  J\nabla_{b}r\right)  }\nabla_{b}r+J\left(  \nabla_{\left(
J\nabla_{b}r\right)  }\nabla_{b}r\right)  \right]  .
\end{array}
\]
Therefore along $\gamma$
\begin{equation}
r_{11}=-r_{1\overline{1}}. \label{302}%
\end{equation}
Moreover, by computing
\[%
\begin{array}
[c]{ccl}%
r_{1} & = & Z_{1}r\\
& = & \frac{1}{\sqrt{2}}\left(  \nabla_{b}r-iJ\nabla_{b}r\right)  r\\
& = & \frac{1}{\sqrt{2}}\left(  \left \vert \nabla_{b}r\right \vert
^{2}-i\left \langle \nabla_{b}r,J\nabla_{b}r\right \rangle \right) \\
& = & \frac{1}{\sqrt{2}}%
\end{array}
\]
and
\[%
\begin{array}
[c]{ccl}%
r_{1\overline{1}} & = & Z_{\overline{1}}Z_{1}r-\Gamma_{\overline{1}1}^{1}%
r_{1}\\
& = & Z_{\overline{1}}Z_{1}r-g_{\theta}\left(  \left[  Z_{\overline{1}}%
,Z_{1}\right]  ,Z_{1}\right)  r_{1}\\
& = & Z_{\overline{1}}Z_{1}r-\frac{1}{\sqrt{2}}g_{\theta}\left(  \left[
Z_{\overline{1}},Z_{1}\right]  ,Z_{1}\right)  ,
\end{array}
\]
we derive that $r_{1\overline{1}}$ is real by the commutation formula.
Therefore, we have
\begin{equation}
r_{0}=0 \label{304a}%
\end{equation}
along the $\nabla$-geodesic $\gamma$. At the point $q$,
\begin{equation}%
\begin{array}
[c]{ccl}%
0 & = & \frac{1}{2}\left(  \left \vert \nabla_{b}r\right \vert ^{2}\right)
_{1\overline{1}}\\
& = & \underset{\alpha}{\sum}\left(  \left \vert r_{\alpha1}\right \vert
^{2}+\left \vert r_{\alpha \overline{1}}\right \vert ^{2}+r_{\alpha1\overline{1}%
}r_{\overline{\alpha}}+r_{\overline{\alpha}1\overline{1}}r_{\alpha}\right) \\
& \geq & \left \vert r_{11}\right \vert ^{2}+\left \vert r_{1\overline{1}%
}\right \vert ^{2}+r_{11\overline{1}}r_{\overline{1}}+r_{\overline{1}%
1\overline{1}}r_{1}\\
& = & 2r_{1\overline{1}}^{2}+\left(  r_{1\overline{1}1}+ir_{10}+R_{11\overline
{1}}^{1}r_{1}\right)  r_{\overline{1}}+\left(  r_{1\overline{1}}%
-ir_{0}\right)  _{\overline{1}}r_{1}\\
& = & 2r_{1\overline{1}}^{2}+\left \langle \nabla_{b}r_{1\overline{1}}%
,\nabla_{b}r\right \rangle _{L_{\theta}}+\frac{1}{2}R_{1\overline{1}%
1\overline{1}}\\
& \geq & 2r_{1\overline{1}}^{2}+\left(  \nabla_{b}r\right)  r_{1\overline{1}%
}\\
& = & 2r_{1\overline{1}}^{2}+\left(  \nabla r\right)  r_{1\overline{1}}\\
& = & 2r_{1\overline{1}}^{2}+\frac{\partial r_{1\overline{1}}}{\partial r}.
\end{array}
\label{310}%
\end{equation}
Here we use the facts that $r_{1}=\frac{1}{\sqrt{2}}$ and $r_{1\overline{1}}$
is real, the equality (\ref{302}), (\ref{304a}), and the commutation formulas
(\ref{201}). Together with the initial condition of $r_{1\overline{1}}$ as $r$
goes to zero, we have%
\[
r_{1\overline{1}}\leq \frac{1}{2r}.
\]
In particular, (\ref{301}) indicates that
\[
\left(  \log r\right)  _{1\overline{1}}=\frac{r_{1\overline{1}}}{r}%
-\frac{\left \vert r_{1}\right \vert ^{2}}{r^{2}}\leq0.
\]
This completes the proof.
\end{proof}

\bigskip

Actually, the similar deductions enable us to derive the substantial CR
sub-Laplacian comparisons below.

\begin{corollary}
\label{c31} If $\left(  M,J,\theta \right)  $ is a complete noncompact
pseudohermitian $\left(  2n+1\right)  $-manifold of vanishing torsion with
nonnegative pseudohermitian Ricci curvature, then we have the CR sub-Laplacian
comparison, for $n\geq2$,%
\begin{equation}
\Delta_{b}r\leq \frac{2^{n}}{r}; \label{321}%
\end{equation}
furthermore, if the pseudohermitian bisectional curvature is nonnegative,
then
\begin{equation}
\Delta_{b}r\leq \frac{2n-1}{r}. \label{322}%
\end{equation}
Finally, the equality holds in (\ref{322}) only if, for any $j\in I_{n}$,
\begin{equation}
R_{1\overline{1}j\overline{j}}=0. \label{333}%
\end{equation}

\end{corollary}

\begin{remark}
When $\left(  M,J,\theta \right)  $ is a complete noncompact pseudohermitian
$3$-manifold of vanishing torsion with nonnegative pseudohermitian Ricci
curvature, it's also easy to derive%
\[
\Delta_{b}r\leq \frac{1}{r}%
\]
from the estimate (\ref{323}). The estimate is sharp in the sense of the
equality holds only if $M$ is flat as in the proof of Theorem \ref{m308} .
\end{remark}

\begin{proof}
By the similar computation as precedes, for any $j\neq1$,
\begin{equation}%
\begin{array}
[c]{ccl}%
0 & = & \frac{1}{2}\left(  \left \vert \nabla_{b}r\right \vert ^{2}\right)
_{j\overline{j}}\\
& = & \underset{\alpha}{\sum}\left(  \left \vert r_{\alpha j}\right \vert
^{2}+\left \vert r_{\alpha \overline{j}}\right \vert ^{2}+r_{\alpha j\overline
{j}}r_{\overline{\alpha}}+r_{\overline{\alpha}j\overline{j}}r_{\alpha}\right)
\\
& \geq & \left \vert r_{j\overline{j}}\right \vert ^{2}+r_{1j\overline{j}%
}r_{\overline{1}}+r_{\overline{1}j\overline{j}}r_{1}\\
& = & r_{j\overline{j}}^{2}+\left(  r_{1\overline{j}j}+ir_{10}+R_{1j\overline
{j}}^{1}r_{1}\right)  r_{\overline{1}}+r_{\overline{1}j\overline{j}}r_{1}\\
& \geq & r_{j\overline{j}}^{2}+\left \langle \nabla_{b}r_{j\overline{j}}%
,\nabla_{b}r\right \rangle _{L_{\theta}}\\
& = & r_{j\overline{j}}^{2}+\left(  \nabla_{b}r\right)  r_{j\overline{j}}\\
& = & r_{j\overline{j}}^{2}+\frac{\partial}{\partial r}r_{j\overline{j}},
\end{array}
\label{315}%
\end{equation}
and the inequality (\ref{301}), it's easy to derive
\[
\Delta_{b}r\leq \frac{2n-1}{r}.
\]
From the inequalities (\ref{310}) and (\ref{315}), it follows that
\begin{equation}%
\begin{array}
[c]{ccl}%
0 & \geq & \left(  2r_{1\overline{1}}^{2}+\frac{\partial r_{1\overline{1}}%
}{\partial r}+\frac{1}{2}R_{1\overline{1}1\overline{1}}\right)  +\underset
{j\neq1}{\sum}\left(  r_{j\overline{j}}^{2}+\frac{\partial}{\partial
r}r_{j\overline{j}}+\frac{1}{2}R_{1\overline{1}j\overline{j}}\right) \\
& \geq & 2^{1-n}\left(  \underset{j}{%
{\displaystyle \sum}
}r_{j\overline{j}}\right)  ^{2}+\frac{\partial}{\partial r}\underset{j}{%
{\displaystyle \sum}
}r_{j\overline{j}}+\frac{1}{2}R_{1\overline{1}}\\
& \geq & 2^{1-n}\left(  \underset{j}{%
{\displaystyle \sum}
}r_{j\overline{j}}\right)  ^{2}+\frac{\partial}{\partial r}\underset{j}{%
{\displaystyle \sum}
}r_{j\overline{j}}%
\end{array}
\label{323}%
\end{equation}
and then%
\[
\Delta_{b}r=2\underset{j}{%
{\displaystyle \sum}
}r_{j\overline{j}}\leq \frac{2^{n}}{r}.
\]

\end{proof}

Now, we could proceed with the proof of the CR three-circle theorem below:

\begin{proof}
(of the Theorem \ref{m302}) First of all, we prove that if $\left(
M,J,\theta \right)  $ admits nonnegative pseudohermitian sectional curvature,
then the CR three-circle theorem holds. On the closure $\overline{A}\left(
p;r_{1},r_{3}\right)  $ of the annulus
\[
A\left(  p;r_{1},r_{3}\right)  =\left \{  x\in M\text{ }|\text{ }r_{1}<r\left(
x\right)  =d\left(  p,x\right)  <r_{3}\right \}
\]
for $0<r_{1}<r_{3}$, we define
\[
F\left(  x\right)  =\left(  \log r_{3}-\log r\left(  x\right)  \right)  \log
M_{f}\left(  r_{1}\right)  +\left(  \log r\left(  x\right)  -\log
r_{1}\right)  \log M_{f}\left(  r_{3}\right)
\]
and%
\[
G\left(  x\right)  =\left(  \log r_{3}-\log r_{1}\right)  \log \left \vert
f\left(  x\right)  \right \vert .
\]
May assume that $M_{f}\left(  r_{1}\right)  <M_{f}\left(  r_{3}\right)  $. Let
$f$ be a CR-holomorphic function on $M$ with
\begin{equation}
f_{0}=0. \label{304}%
\end{equation}
It suffices to claim
\[
G\leq F
\]
on $\overline{A}\left(  p;r_{1},r_{3}\right)  $. It's clear that $G\leq F$ on
the boundary $\partial A\left(  p;r_{1},r_{3}\right)  $ of the annulus
$A\left(  p;r_{1},r_{3}\right)  .$ Suppose that $G\left(  x\right)  >F\left(
x\right)  $ for some interior point $x$ in $A\left(  p;r_{1},r_{3}\right)  $,
then we could choose a point $q\in A\left(  p;r_{1},r_{3}\right)  $ such that
the function $\left(  G-F\right)  $ attains the maximum value at $q$.

If $q\notin Cut\left(  p\right)  $, then%
\[
i\partial_{b}\overline{\partial}_{b}\left(  G-F\right)  \left(  q\right)
\leq0
\]
by observing that the inequality $i\partial_{b}\overline{\partial}_{b}\left(
G-F\right)  \left(  q\right)  >0$, which says that $\left(  G-F\right)
_{\alpha \overline{\beta}}$ is positive definite, implies the positivity of the
sub-Lplacian $\Delta_{b}\left(  G-F\right)  \left(  q\right)  >0$ (this
contradicts that $q$ is a maximum point of $\left(  G-F\right)  $). In
particular,
\begin{equation}
\left(  G-F\right)  _{1\overline{1}}\left(  q\right)  \leq0 \label{303}%
\end{equation}
where $Z_{1}=\frac{1}{\sqrt{2}}\left(  \nabla_{b}r-iJ\nabla_{b}r\right)  $.
Note that $\left(  G-F\right)  _{0}(q)=0$ due to (\ref{304}) and (\ref{304a}).

On the other hand, it follows from \cite{fow} or \cite{chl} that there is a
transverse K\"{a}hler structure at the point $q$ and we denoted such local
coordinates in some open neighborhood $U$ of the point $q$ by $\left \{
z_{\alpha},x\right \}  _{\alpha \in I_{n}}$ with $T=\frac{\partial}{\partial x}$
and
\[
\underset{}{Z_{1}\left(  q\right)  =\left(  \frac{\partial}{\partial z_{1}%
}-\theta \left(  \frac{\partial}{\partial z_{1}}\right)  T\right)  }|_{q}.
\]

Restrict to the leaf space $\widetilde{D}=\left[  x=0\right]  $ and write the
point $y$ in $U$ as $\left(  \widetilde{y},x\right)  $. It's clear that
$q=\left(  \widetilde{q},0\right)  $. Hereafter the quantity with the tilde
means such one lies in the slice $\widetilde{D}$. This enables us to transfer
the local property of the K\"{a}hler manifolds to the CR manifolds. Let
$\widetilde{G}$ and $\widetilde{F}$ denote the restrictions of $G$ and $F$ to
the leaf space $\widetilde{D}$. So $\widetilde{q}$ is a maximum point of
$\left(  \widetilde{G}-\widetilde{F}\right)  $ and $\widetilde{f}%
=f|_{\widetilde{D}}$ is a holomorphic function on $U\cap \widetilde{D}$.
Because $\left(  \widetilde{G}-\widetilde{F}\right)  $ attains the maximum
value at $\widetilde{q}$ (this implies that $\left \vert \widetilde{f}\left(
\widetilde{q}\right)  \right \vert \neq0$), the Poincar\'{e}-Lelong equation%
\[
\frac{i}{2\pi}\partial \overline{\partial}\log \left \vert \widetilde
{f}\right \vert ^{2}=\left[  D\left(  \widetilde{f}\right)  \right]
\]
gives
\[
\widetilde{G}_{1\overline{1}}\left(  \widetilde{q}\right)  =0
\]
where $D\left(  \widetilde{f}\right)  $ is the divisor of $\widetilde{f}$.
However, the condition (\ref{304}) implies that $G$ is independent of the
characteristic direction $T.$ So we have the equality
\[
G_{1\overline{1}}\left(  q\right)  =0.
\]

This and Lemma \ref{m301} indicate that
\[
\left(  G-F\right)  _{1\overline{1}}\left(  q\right)  \geq0;
\]
however, it's not enough to obtain the contradiction from (\ref{303}). So we
take a modified function $F_{\epsilon}$ for replacing the original function
$F$ so that it enables us to get the contradiction. Set
\[
F_{\epsilon}\left(  x\right)  =a_{\epsilon}\log \left(  r\left(  x\right)
-\epsilon \right)  +b_{\epsilon}%
\]
for any sufficiently small positive number $\epsilon$. Here the two constants
$a_{\epsilon}$ and $b_{\epsilon}$ are determined by the following equations:%
\[
F_{\epsilon}\left(  r_{j}\right)  =F_{\epsilon}\left(  \partial B\left(
p,r_{j}\right)  \right)  =\log \left(  \frac{r_{3}}{r_{1}}\right)  \log
M_{f}\left(  r_{j}\right)
\]
for $j=1,3$. It's apparent to see $F_{\epsilon}\rightarrow F$ on the annulus
$A\left(  p;r_{1},r_{3}\right)  $ as $\epsilon \rightarrow0^{+}$ and
$a_{\epsilon}>0$. Let $q_{\epsilon}$ be a maximum point in $A\left(
p;r_{1},r_{3}\right)  $ of the function $\left(  G-F_{\epsilon}\right)  $.
With the same deduction, we have
\begin{equation}
\left(  G-F_{\epsilon}\right)  _{1\overline{1}}\left(  q_{\epsilon}\right)
\leq0 \label{313}%
\end{equation}
and%
\[
G_{1\overline{1}}\left(  q_{\epsilon}\right)  =0.
\]
From Lemma \ref{m301} again, we obtain%
\[
\left(  \log \left(  r-\epsilon \right)  \right)  _{1\overline{1}}\left(
q_{\epsilon}\right)  =\frac{r_{1\overline{1}}}{\left(  r-\epsilon \right)
}-\frac{1}{2\left(  r-\epsilon \right)  ^{2}}<0.
\]
Accordingly, we have
\[
\left(  G-F_{\epsilon}\right)  _{1\overline{1}}\left(  q_{\epsilon}\right)
>0\text{.}%
\]
It contradicts with the inequality (\ref{313}). This indicates that
\[
\left(  G-F_{\epsilon}\right)  \leq0
\]
in the annulus $A\left(  p;r_{1},r_{3}\right)  $.

If $q_{\epsilon}\in Cut\left(  p\right)  $, then we adopt the trick of Calabi
as follows. Choose a number $\epsilon_{1}\in \left(  0,\epsilon \right)  $ and
the point $p_{1}$ lying on the minimal $D$-geodesic from $p$ to $q_{\epsilon}$
with $d\left(  p,p_{1}\right)  =\epsilon_{1}$. Set%
\[
\widehat{r}\left(  x\right)  =d\left(  p_{1},x\right)
\]
and consider the slight modification of the function $F_{\epsilon}\left(
x\right)  $%
\[
F_{\epsilon,\epsilon_{1}}\left(  x\right)  =a_{\epsilon}\log \left(
\widehat{r}+\epsilon_{1}-\epsilon \right)  +b_{\epsilon}\text{.}%
\]

It's not hard to observe that $F_{\epsilon}\left(  q_{\epsilon}\right)
=F_{\epsilon,\epsilon_{1}}\left(  q_{\epsilon}\right)  $ and $F_{\epsilon}\leq
F_{\epsilon,\epsilon_{1}}$; hence, we know $\left(  G-F_{\epsilon,\epsilon
_{1}}\right)  $ also attains the maximum value at $q_{\epsilon}$. Then,
applying the similar argument as precedes, we still have%
\[
G-F_{\epsilon,\epsilon_{1}}\leq0
\]
in the annulus $A\left(  p;r_{1},r_{3}\right)  $. Letting $\epsilon
_{1}\rightarrow0^{+}$, then $\epsilon \rightarrow0^{+}$, the validity of the CR
three circle theorem is settled.

As for the monotonicity of (\ref{314}), it's easily derived by taking the
$3$-tuples $\left(  r_{1},r_{2},kr_{2}\right)  $ and $\left(  r_{1}%
,kr_{1},kr_{2}\right)  $ into the convexity of the CR three circle theorem for
$0<r_{1}\leq r_{2}<+\infty$.
\end{proof}

\bigskip

It's clear that we have the following sharp monotonicity:

\begin{proposition}
\label{p1}If $\left(  M,J,\theta \right)  $ is a complete noncompact
pseudohermitian $\left(  2n+1\right)  $-manifold on which CR three-circle
theorem holds and $ord_{p}f\geq k\geq1$ for some $f\in \mathcal{O}^{CR}(M)$ and
$p\in M$, then
\[
\frac{M_{f}\left(  r\right)  }{r^{k}}%
\]
is increasing in terms of $r$.
\end{proposition}

\begin{proof}
Let $0<r_{2}\leq r_{3}<+\infty$. Since the vanishing order of $f$ at $p$ is at
least $k$, for any $\epsilon>0$, there is a sufficiently small number
$0<r_{1}<r_{2}$, such that
\begin{equation}
\log M_{f}\left(  r_{1}\right)  \leq \log M_{f}\left(  r_{3}\right)  +\left(
k-\epsilon \right)  \log \frac{r_{1}}{r_{3}}. \label{317}%
\end{equation}
Substituting (\ref{317}) into the inequality (\ref{101}), we get%
\[
\frac{M_{f}\left(  r_{2}\right)  }{r_{2}^{k-\epsilon}}\leq \frac{M_{f}\left(
r_{3}\right)  }{r_{3}^{k-\epsilon}}.
\]
The proposition is accomplished by letting $\epsilon \rightarrow0^{+}$.
\end{proof}

\bigskip

Therefore, Theorem \ref{m302} and Proposition \ref{p1} imply

\begin{corollary}
\label{m310}If $\left(  M,J,\theta \right)  $ is a complete noncompact
pseudohermitian $\left(  2n+1\right)  $-manifold of vanishing torsion with
nonnegative pseudohermitian sectional curvature, and $ord_{p}f\geq k\geq1$ for
some $f\in \mathcal{O}^{CR}(M)$ and $p\in M$, then
\[
\frac{M_{f}\left(  r\right)  }{r^{k}}%
\]
is increasing in terms of $r$.
\end{corollary}

Before giving a variety of applications of the CR three-circle theorem, we
need some notations about the CR-holomorphic functions of polynomial growth.

\begin{definition}
Let $\left(  M,J,\theta \right)  $ be a complete noncompact pseudohermitian
$\left(  2n+1\right)  $-manifold. We denote the collections%
\[
\left \{  f\in \mathcal{O}^{CR}(M)\left \vert \text{ }\underset{r\rightarrow
+\infty}{\lim \sup}\frac{M_{f}\left(  r\right)  }{r^{d}}<+\infty \right.
\right \}
\]
and
\[
\left \{  f\in \mathcal{O}^{CR}(M)\left \vert \text{ }\underset{r\rightarrow
+\infty}{\lim \inf}\frac{M_{f}\left(  r\right)  }{r^{d}}<+\infty \right.
\right \}
\]
by $\widetilde{\mathcal{O}}_{d}^{CR}(M)$ and $\widehat{\mathcal{O}}_{d}%
^{CR}(M),$ respectively.
\end{definition}

It's clear that $\widetilde{\mathcal{O}}_{d}^{CR}(M)\subseteq \mathcal{O}%
_{d}^{CR}(M)\cap \widehat{\mathcal{O}}_{d}^{CR}(M)$.

\begin{proposition}
\label{m304}If $\left(  M,J,\theta \right)  $ is a complete noncompact
pseudohermitian $\left(  2n+1\right)  $-manifold of vanishing torsion with
nonnegative pseudohermitian sectional curvature, then $f\in \widehat
{\mathcal{O}}_{d}^{CR}(M)$ if and only if $\frac{M_{f}\left(  r\right)
}{r^{d}}$ is decreasing with respect to $r$ for any $f\in \mathcal{O}^{CR}(M)$.
\end{proposition}

\begin{proof}
It's straightforward to see that the sufficient part holds. On the other hand,
let $0<r_{1}\leq r_{2}<+\infty$, by the assumption that
\[
f\in \widehat{\mathcal{O}}_{d}^{CR}(M),
\]
for any positive number $\epsilon$, there's a sequence $\left \{  \lambda
_{j}\right \}  \nearrow+\infty$ such that
\[
\log M_{f}\left(  \lambda_{j}\right)  \leq \log M_{f}\left(  r_{1}\right)
+\left(  d+\epsilon \right)  \log \lambda_{j}.
\]
From Theorem \ref{m302}, by taking $r_{3}=\lambda_{j}$, it follows that%
\[
\log M_{f}\left(  r_{2}\right)  \leq \log M_{f}\left(  r_{1}\right)  +\left(
d+\epsilon \right)  \log \frac{r_{2}}{r_{1}}.
\]
Let $\epsilon$ go to zero, the necessary part follows.
\end{proof}

\bigskip

From the last proposition, it's easy to deduce the conclusion below.

\begin{corollary}
\label{m305}If $\left(  M,J,\theta \right)  $ is a complete noncompact
pseudohermitian $\left(  2n+1\right)  $-manifold of vanishing torsion with
nonnegative pseudohermitian sectional curvature, then%
\[
\widetilde{\mathcal{O}}_{d}^{CR}(M)=\widehat{\mathcal{O}}_{d}^{CR}(M).
\]

\end{corollary}

In addition, we have the asymptotic property for the degree of CR-holomorphic
functions of polynomial growth as follows:

\begin{corollary}
\label{m306}Let $\left(  M,J,\theta \right)  $ be a complete noncompact
pseudohermitian $\left(  2n+1\right)  $-manifold of vanishing torsion with
nonnegative pseudohermitian sectional curvature. If
\[
f\in \widetilde{\mathcal{O}}_{d+\epsilon}^{CR}(M)
\]
for any number $\epsilon>0$, then%
\[
f\in \widetilde{\mathcal{O}}_{d}^{CR}(M).
\]

\end{corollary}

\begin{proof}
From Proposition \ref{m304} and Corollary \ref{m305}, we know
\[
\frac{M_{f}\left(  r\right)  }{r^{d+\epsilon}}%
\]
is decreasing with respect to $r$ for any positive number $\epsilon$. By
acontradiction, it's easy to validate the monotonicity of
\[
\frac{M_{f}\left(  r\right)  }{r^{d}}.
\]
Then, by the last two corollaries, we have%
\[
f\in \widetilde{\mathcal{O}}_{d}^{CR}(M).
\]

\end{proof}

\bigskip

From Corollary \ref{m306}, we know

\begin{corollary}
\label{m307}If $\left(  M,J,\theta \right)  $ is a complete noncompact
pseudohermitian $\left(  2n+1\right)  $-manifold of vanishing torsion with
nonnegative pseudohermitian sectional curvature, then%
\[
\mathcal{O}_{d}^{CR}(M)=\widetilde{\mathcal{O}}_{d}^{CR}(M).
\]

\end{corollary}

As an application, we could recover and generalize the CR sharp dimension
estimate in \cite{chl} under the assumption of nonnegative pseudohermitian
sectional curvature instead of nonnegative pseudohermitian bisectional curvature.

\begin{proof}
(of Theorem \ref{m308}) Suppose on the contrary, i.e.
\[
\dim_{%
\mathbb{C}
}\left(  \mathcal{O}_{d}^{CR}(M)\right)  >\dim_{%
\mathbb{C}
}\left(  \mathcal{O}_{d}^{CR}(\mathbb{H}^{n})\right)
\]
for some positive integer $d\in%
\mathbb{N}
$. Then for any point $p\in M$, there's a nonzero CR-holomorphic function $f$
of polynomial growth of degree at most $d$ with
\[
ord_{p}f\geq d+1
\]
Concerning the existence of such function $f$, one could refer to the proof of
Theorem 1.1 in \cite[(5.17)]{chl} which is just a method of the
Poincar\'{e}-Siegel argument via linear algebra (\cite{m1}). Therefore, we
have%
\[
\underset{r\rightarrow0^{+}}{\lim}\frac{M_{f}\left(  r\right)  }{r^{d}}=0.
\]
However, this contradicts with the monotonicity of the function $\frac
{M_{f}\left(  r\right)  }{r^{d}}$ as in Proposition \ref{m304}. Hence, the
sharp dimension estimate holds. As for the rigidity part, we just claim that
if
\begin{equation}
\dim_{%
\mathbb{C}
}\left(  \mathcal{O}_{d}^{CR}(M)\right)  =\dim_{%
\mathbb{C}
}\left(  \mathcal{O}_{d}^{CR}(\mathbb{H}^{n})\right)  , \label{319}%
\end{equation}
then $\left(  M,J,\theta \right)  $ is CR-isomorphic to $(2n+1)$-dimensional
Heisenberg group $\mathbb{H}^{n}$. From Proposition 4.1 in \cite{t} (or
Theorem 7.15 in \cite{b}), it suffices to show that $M$ has constant
$J$-holomorphic sectional curvature $-3$. From the equation, for any $p\in M$,
$Z\in \left(  T_{1,0}M\right)  _{p}$ with $\left \vert Z\right \vert =1$,%
\[%
\begin{array}
[c]{ccl}%
R^{\theta}\left(  Z,\overline{Z},Z,\overline{Z}\right)  & = & R\left(
Z,\overline{Z},Z,\overline{Z}\right)  +g_{\theta}\left(  \left(
Z\wedge \overline{Z}\right)  Z,\overline{Z}\right)  +2d\theta \left(
Z,\overline{Z}\right)  g_{\theta}\left(  JZ,\overline{Z}\right) \\
& = & R\left(  Z,\overline{Z},Z,\overline{Z}\right)  -3,
\end{array}
\]
the proof of the rigidity part is completed if we justify that the
pseudohermitian sectional curvature vanishes. Adopting the notations as in
Theorem \ref{m302}, we just claim that, for simplification,%
\begin{equation}
R\left(  Z_{1},\overline{Z}_{1},Z_{1},\overline{Z}_{1}\right)  \left(
p\right)  =0 \label{320}%
\end{equation}
where
\[
Z_{1}=\frac{\partial}{\partial z_{1}}-\theta \left(  \frac{\partial}{\partial
z_{1}}\right)  T.
\]
The equality (\ref{319}) enables us to see that there is a function
$f\in \mathcal{O}_{d}^{CR}(M)$ such that%
\[
f\left(  z_{1},...,z_{n},x\right)  =z_{1}^{d}+O\left(  r^{d+1}\right)
\]
locally. This impies that
\[
ord_{p}f=d.
\]
Therefore, from Corollary \ref{m310} and Proposition \ref{m304}, we obtain%
\[
\frac{M_{f}\left(  r\right)  }{r^{d}}%
\]
is constant. In the proofs of Theorem \ref{m302} and Lemma \ref{m301}, it's
not difficult to find that $G-F$ attains the maximum value $0$ on $\partial
B\left(  p,r\right)  $ at the point $q\left(  r\right)  $ for any positive
number $r$ and then
\[
R\left(  \nabla_{b}r,J\nabla_{b}r,\nabla_{b}r,J\nabla_{b}r\right)  \left(
q\left(  r\right)  \right)  =0.
\]
From the definition of the chosen function $f$, we could take a subsequence
$\left \{  \left(  \nabla_{b}r\right)  \left(  q\left(  r_{j}\right)  \right)
\right \}  _{j\in%
\mathbb{N}
}$ such that its limit, as $r\rightarrow0^{+}$, lies in the tangent space at
$p$ spanned by $\frac{\partial}{\partial z_{1}}|_{p}$ and $T|_{p}$. Then the
equality (\ref{320}) holds by the formula (\ref{324}). Accordingly, this
theorem is accomplished.
\end{proof}

\section{An extension of CR Three-Circle Theorem}

Subsequently, we will give the proof of the CR three-circle theorem when the
pseudohermitian sectional curvature is bounded from below by a function.

\begin{proof}
(of Theorem \ref{m303}) Although this proof is similar to the one of the CR
three-circle theorem, we give its proof for completeness. Here we adopt the
same notations as in the proof of Theorem \ref{m302}. Define%
\[
F\left(  x\right)  =\left(  h\left(  r_{3}\right)  -h\left(  r\left(
x\right)  \right)  \right)  \log M_{f}\left(  r_{1}\right)  +\left(  h\left(
r\left(  x\right)  \right)  -h\left(  r_{1}\right)  \right)  \log M_{f}\left(
r_{3}\right)
\]
and%
\[
G\left(  x\right)  =\left(  h\left(  r_{3}\right)  -h\left(  r_{1}\right)
\right)  \log \left \vert f\left(  x\right)  \right \vert
\]
on the annulus $A\left(  p;r_{1},r_{3}\right)  $ for $0<r_{1}<r_{3}<+\infty$.
We still assume that
\begin{equation}
M_{f}\left(  r_{1}\right)  <M_{f}\left(  r_{3}\right)  . \label{309}%
\end{equation}
It's clear that $G\leq F$ on the boundary $\partial A\left(  p;r_{1}%
,r_{3}\right)  $ by (\ref{307}). It suffices to show that%
\[
G\leq F
\]
on the annulus $A\left(  p;r_{1},r_{3}\right)  $ to reach our first conclusion
by the maximum principle. Suppose that $G\left(  x\right)  >F\left(  x\right)
$ for some interior point $x$ in $A\left(  p;r_{1},r_{3}\right)  $, then we
could choose a point $q\in A\left(  p;r_{1},r_{3}\right)  $ such that the
function $\left(  G-F\right)  $ attains the maximum value at $q$.

If $q\notin Cut\left(  p\right)  $, then
\begin{equation}
\left(  G-F\right)  _{1\overline{1}}\left(  q\right)  \leq0. \label{312}%
\end{equation}
With the same deduction in Theorem \ref{m302}, we have
\[
G_{1\overline{1}}\left(  q\right)  =0
\]
from the Poincar\'{e}-Lelong equation and the fact that $f\left(  q\right)
\neq0$ with the help of the transverse K\"{a}hler structure. Due to the fact
that $h\left(  r\right)  \thicksim \log r$ as $r\rightarrow0^{+}$, we could
define
\[
F_{\epsilon}\left(  x\right)  =a_{\epsilon}\log \left(  e^{h\left(  r\right)
}-\epsilon \right)  +b_{\epsilon}%
\]
for any sufficiently small number $\epsilon>0$ and the two constants
$a_{\epsilon}$ and $b_{\epsilon}$ are restricted by the following equations%
\[
F_{\epsilon}\left(  r_{j}\right)  =\left(  h\left(  r_{3}\right)  -h\left(
r_{1}\right)  \right)  \log M_{f}\left(  r_{j}\right)
\]
for $j=1,3$. It's obvious that $F_{\epsilon}\rightarrow F$ on the annulus
$A\left(  p;r_{1},r_{3}\right)  $ as $\epsilon \rightarrow0^{+}$. Due to the
inequality (\ref{307}) and the assumption (\ref{309}), we see that
$a_{\epsilon}>0$. Denote by $q_{\epsilon}$ a maximum point in $A\left(
p;r_{1},r_{3}\right)  $ of the function $\left(  G-F_{\epsilon}\right)  $ and
modify the point $q$ into the point $q_{\epsilon}$. From (\ref{306}),
(\ref{310}), (\ref{305}), and the initial condition $u\left(  r\right)
\thicksim \frac{1}{2r}$ as $r\rightarrow0^{+}$ imply that%
\begin{equation}
r_{1\overline{1}}\leq u\left(  r\right)  . \label{311}%
\end{equation}
By the hypotheses (\ref{307}), (\ref{308}) and the inequality (\ref{311}), we
get%
\[%
\begin{array}
[c]{l}%
\left(  \log \left(  e^{h\left(  r\right)  }-\epsilon \right)  \right)
_{1\overline{1}}\left(  q_{\epsilon}\right) \\
=\frac{-\epsilon e^{h}\left(  h^{\prime}\right)  ^{2}\left \vert r_{1}%
\right \vert ^{2}+\left(  e^{h\left(  r\right)  }-\epsilon \right)  \left(
e^{h}h^{\prime \prime}\left \vert r_{1}\right \vert ^{2}+e^{h}h^{\prime
}r_{1\overline{1}}\right)  }{\left(  e^{h\left(  r\right)  }-\epsilon \right)
^{2}}\\
\leq-\frac{\epsilon e^{h}\left(  h^{\prime}\right)  ^{2}}{2\left(  e^{h\left(
r\right)  }-\epsilon \right)  ^{2}}\\
<0
\end{array}
\]
for sufficiently small positive number $\epsilon.$ It yields that
\[
\left(  G-F_{\epsilon}\right)  _{1\overline{1}}\left(  q_{\epsilon}\right)
>0.
\]
This contradicts with (\ref{312}). Therefore we obtain%
\[
G\leq F_{\epsilon}%
\]
for sufficeintly small number $\epsilon>0$. If $q\in Cut\left(  p\right)  $,
then, by the trick of Calabi again, let $\epsilon_{1}\in \left(  0,\epsilon
\right)  $ and the point $p_{1}$ lying on the minimal $D$-geodesic from $p$ to
$q_{\epsilon}$ with $d\left(  p,p_{1}\right)  =\epsilon_{1}$. Set%
\[
\widehat{r}\left(  x\right)  =d\left(  p_{1},x\right)
\]
and consider the modified function $F_{\epsilon,\epsilon_{1}}\left(  x\right)
$ of the function $F_{\epsilon}\left(  x\right)  $%
\[
F_{\epsilon,\epsilon_{1}}\left(  x\right)  =a_{\epsilon}\log \left(
e^{h\left(  \widehat{r}+\epsilon_{1}\right)  }-\epsilon \right)  +b_{\epsilon
}\text{.}%
\]
Due to the monotonicity of the function $h$, we see that $F_{\epsilon}\leq
F_{\epsilon,\epsilon_{1}}$. It's clear that $F_{\epsilon}\left(  q_{\epsilon
}\right)  =F_{\epsilon,\epsilon_{1}}\left(  q_{\epsilon}\right)  $. So the
point $q_{\epsilon}$ is still a maximum point of $\left(  G-F_{\epsilon
,\epsilon_{1}}\right)  $. Set
\[
\widehat{Z}_{1}=\frac{1}{\sqrt{2}}\left(  \nabla_{b}\widehat{r}-iJ\nabla
_{b}\widehat{r}\right)  .
\]
By observing the expansion of $\left(  \log \left(  e^{h\left(  \widehat
{r}+\epsilon_{1}\right)  }-\epsilon \right)  \right)  _{\widehat{1}%
\overline{\widehat{1}}}\left(  q_{\epsilon}\right)  $, (\ref{307}),
(\ref{308}), and the continuity of the pseudohermitian sectional curvature
imply that
\[
\left(  F_{\epsilon,\epsilon_{1}}\right)  _{\widehat{1}\overline{\widehat{1}}%
}\left(  q_{\epsilon}\right)  <0
\]
for sufficiently small $\epsilon_{1}>0$ for fixed $\epsilon$. Here the
property that the pseudohermitian sectional curvature is continuous is
utilized to obtain the estimate%
\[
\widehat{r}_{\widehat{1}\overline{\widehat{1}}}\leq u+\epsilon^{\prime}%
\]
for small positive error $\epsilon^{\prime}=\epsilon^{\prime}\left(
\epsilon_{1}\right)  $. \ Then $\left(  G-F_{\epsilon,\epsilon_{1}}\right)
_{\widehat{1}\overline{\widehat{1}}}\left(  q_{\epsilon}\right)  >0.$ However,
it contradicts with the fact that $\left(  G-F_{\epsilon,\epsilon_{1}}\right)
$ attains a maximum point at $q_{\epsilon}$. Accordingly, the inequality%
\[
G\leq F_{\epsilon,\epsilon_{1}}%
\]
holds. Letting $\epsilon_{1}\rightarrow0^{+}$, then $\epsilon \rightarrow0^{+}%
$, we have
\[
G\leq F
\]
on the annulus $A\left(  p;r_{1},r_{3}\right)  $. \ Because $ord_{p}\left(
f\right)  =d$ and $h\left(  r\right)  \thicksim \log r$ as $r\rightarrow0^{+}$,
then we have, for any $\epsilon>0$,%
\[
\log M_{f}\left(  r_{1}\right)  \leq \log M_{f}\left(  r_{2}\right)  +\left(
d-\epsilon \right)  \left(  h\left(  r_{1}\right)  -h\left(  r_{2}\right)
\right)
\]
for sufficiently small positive number $r_{1}$ and $r_{1}<r_{2}$. By the
convexity of $\log M_{f}\left(  r\right)  $ with respect to the function
$h\left(  r\right)  $%
\[
\log M_{f}\left(  r\right)  \leq \frac{h\left(  r_{2}\right)  -h\left(
r\right)  }{h\left(  r_{2}\right)  -h\left(  r_{1}\right)  }\log M_{f}\left(
r_{1}\right)  +\frac{h\left(  r\right)  -h\left(  r_{1}\right)  }{h\left(
r_{2}\right)  -h\left(  r_{1}\right)  }\log M_{f}\left(  r_{2}\right)
\]
for $r_{1}\leq r\leq r_{2}$, we obtain the monotonicity of $\frac{M_{f}\left(
r\right)  }{\exp \left(  dh\left(  r\right)  \right)  }$. This completes the proof.
\end{proof}

\bigskip

Choosing the functions $g\left(  r\right)  =-1$, $u\left(  r\right)
=\frac{\left(  e^{2r}+1\right)  }{2\left(  e^{2r}-1\right)  }$, and $h\left(
r\right)  =\log \frac{e^{r}-1}{e^{r}+1}$ in Theorem \ref{m303}, we have the
following consequence:

\begin{corollary}
If $\left(  M,J,\theta \right)  $ is a complete noncompact pseudohermitian
$\left(  2n+1\right)  $-manifold of vanishing torsion with the pseudohermitian
sectional curvature bounded from below by $-1$, $f\in O^{CR}\left(  M\right)
$, then $\log M_{f}\left(  r\right)  $ is convex with respect to the function
$\log \frac{e^{r}-1}{e^{r}+1}$. In particular, $\frac{M_{f}\left(  r\right)
}{\left(  \frac{e^{r}-1}{e^{r}+1}\right)  ^{d}}$ is increasing for
$ord_{p}\left(  f\right)  =d$.
\end{corollary}

Similarly, choosing the functions $g\left(  r\right)  =1$, $u\left(  r\right)
=\frac{1}{2}\cot r$, and $h\left(  r\right)  =\log \tan \frac{r}{2}$ in Theorem
\ref{m303}, we obtain

\begin{corollary}
If $\left(  M,J,\theta \right)  $ is a complete noncompact pseudohermitian
$\left(  2n+1\right)  $-manifold of vanishing torsion with the pseudohermitian
sectional curvature bounded from below by $1$, $f\in O^{CR}\left(  B\left(
p,R\right)  \right)  $, then $\log M_{f}\left(  r\right)  $ is convex with
respect to the function $\log \tan \frac{r}{2}$. In particular, $\frac
{M_{f}\left(  r\right)  }{\left(  \tan \frac{r}{2}\right)  ^{d}}$ is increasing
for $ord_{p}\left(  f\right)  =d$.
\end{corollary}

With the help of Theorem \ref{m303}, we have the dimension estimate when the
pseudohermitian sectional curvature is asymptotically nonnegative.

\begin{proof}
(of Theorem \ref{m309}) Although the proof is almost the same as in
\cite{liu1}, we give its proof for completeness. May assume $\epsilon<\frac
{1}{2}$. Choose
\[
u\left(  r\right)  =\frac{1}{2r}+\frac{A}{\left(  1+r\right)  ^{1+\epsilon}};
\]
hence, the inequality holds%
\[
2u^{2}+u^{\prime}-\frac{1}{2}\frac{A}{\left(  1+r\right)  ^{2+\epsilon}}%
\geq0.
\]
Suppose $h\left(  r\right)  $ is the solution to the equation%
\[
\left \{
\begin{array}
[c]{l}%
\frac{1}{2}h^{\prime \prime}+h^{\prime}u=0,\\
\underset{r\rightarrow0^{+}}{\lim}\frac{\exp \left(  h\left(  r\right)
\right)  }{r}=1,
\end{array}
\right.
\]
then
\[
h^{\prime}\left(  r\right)  =\frac{\exp \left(  \frac{2A}{\epsilon \left(
1+r\right)  ^{\epsilon}}\right)  }{r}\exp \left(  -\frac{2A}{\epsilon}\right)
.
\]
Accordingly,
\[
h\left(  r\right)  \geq \exp \left(  -\frac{2A}{\epsilon}\right)  \log r+C
\]
for any number $r\geq1$. Here $C=C\left(  A,\epsilon \right)  $. Theorem
\ref{m303} implies that if the vanishing order $ord_{p}\left(  f\right)  $ of
$f\in \mathcal{O}^{CR}\left(  M\right)  $ at $p$ is equal to $d$, then
$\frac{M_{f}\left(  r\right)  }{\exp \left(  dh\left(  r\right)  \right)  }$ is
increasing with respect to $r$. So%
\begin{equation}
M_{f}\left(  r\right)  \geq \exp \left(  dh\left(  r\right)  \right)
\underset{s\rightarrow0^{+}}{\lim}\frac{M_{f}\left(  s\right)  }{\exp \left(
dh\left(  s\right)  \right)  }\geq C_{1}r^{d\exp \left(  -\frac{2A}{\epsilon
}\right)  } \label{318}%
\end{equation}
where
\[
C_{1}=\exp \left(  Cd\right)  \underset{s\rightarrow0^{+}}{\lim}\frac
{M_{f}\left(  s\right)  }{\exp \left(  dh\left(  s\right)  \right)  }.
\]
Consider the Poincar\'{e}-Siegel map%
\[%
\begin{array}
[c]{c}%
\Phi:\mathcal{O}_{d}^{CR}\left(  M\right)  \longrightarrow%
\mathbb{C}
^{q\left(  \left[  d\exp \left(  \frac{2A}{\epsilon}\right)  \right]  \right)
}\\
\text{ \  \  \  \  \  \  \  \  \  \  \  \  \  \  \  \  \  \  \  \  \  \  \  \  \ }f\longmapsto \left(
D^{\alpha}f\right)  _{\left \vert \alpha \right \vert \leq \left[  d\exp \left(
\frac{2A}{\epsilon}\right)  \right]  }%
\end{array}
\]
where $q\left(  m\right)  =\binom{n+m}{n}$ for any $m\in%
\mathbb{N}
$ and $\left[  a\right]  $ denotes the greatest integer less than or equal to
$a$. We would claim that $\Phi$ is injective; for if $0\neq f\in
\mathcal{O}_{d}^{CR}\left(  M\right)  $ and $D^{\alpha}f=0$ for any
$\left \vert \alpha \right \vert \leq d^{\prime}=\left[  d\exp \left(  \frac
{2A}{\epsilon}\right)  \right]  $, then
\[
ord_{p}\left(  f\right)  \geq d^{\prime}+1.
\]
Hence, by (\ref{318}), we obtain%
\[
M_{f}\left(  r\right)  \geq C_{1}r^{\left(  1+d^{\prime}\right)  \exp \left(
-\frac{2A}{\epsilon}\right)  };
\]
however, this contradicts with the fact%
\[
f\in \mathcal{O}_{d}^{CR}\left(  M\right)  .
\]
Therefore we have the dimension estimate%
\[
\dim_{%
\mathbb{C}
}\left(  \mathcal{O}_{d}^{CR}\left(  M\right)  \right)  \leq \dim_{%
\mathbb{C}
}\left(  \mathcal{O}_{\left[  d\exp \left(  \frac{2A}{\epsilon}\right)
\right]  }^{CR}\left(  H^{n}\right)  \right)  =C\left(  \epsilon,A\right)
d^{n}%
\]
for any $d\in%
\mathbb{N}
$. If $d\leq e^{-\frac{3A}{\epsilon}}$, then
\[
\dim_{%
\mathbb{C}
}\left(  \mathcal{O}_{d}^{CR}\left(  M\right)  \right)  \leq \dim_{%
\mathbb{C}
}\left(  \mathcal{O}_{d\exp \left(  \frac{2A}{\epsilon}\right)  }\left(
\mathbb{C}
^{n}\right)  \right)  \leq \dim_{%
\mathbb{C}
}\left(  \mathcal{O}_{\exp \left(  -\frac{A}{\epsilon}\right)  }\left(
\mathbb{C}
^{n}\right)  \right)  =1.
\]
Last, if $\frac{A}{\epsilon}\leq \frac{1}{4d}$, then
\[
d\exp \left(  \frac{2A}{\epsilon}\right)  <d+1
\]
and%
\[
\dim_{%
\mathbb{C}
}\left(  \mathcal{O}_{d}^{CR}\left(  M\right)  \right)  \leq \dim_{%
\mathbb{C}
}\left(  \mathcal{O}_{d}^{CR}\left(  \mathbb{H}^{n}\right)  \right)  .
\]
This theorem is accomplished.
\end{proof}

\bigskip

It's not difficult to observe that Theorem \ref{m309} includes the case when
the pseudohermitian sectional curvature is nonnegative outside a compact set
as follows:

\begin{corollary}
Let $\left(  M,J,\theta \right)  $ is a complete noncompact pseudohermitian
$\left(  2n+1\right)  $-manifold of vanishing torsion of which the
pseudohermitian sectional curvature is nonnegative outside a compact subset
$K$ and is bounded from below by $-a$ for some $a>0$ on $M$. If $\lambda
=a\left(  d_{c}\left(  K\right)  \right)  ^{2}$ where $d_{c}\left(  K\right)
$ denotes the diameter of $K$, then there is a positive constant $C\left(
\lambda,n\right)  $ such that
\[
\dim_{%
\mathbb{C}
}\left(  \mathcal{O}_{d}^{CR}\left(  M\right)  \right)  \leq C\left(
\lambda,n\right)  d^{n}%
\]
for any positive integer $d$. For any $d\in%
\mathbb{N}
$, there is a positive number $\epsilon \left(  d\right)  $ such that if
$\lambda \leq \epsilon \left(  d\right)  $, then we have
\[
\dim_{%
\mathbb{C}
}\left(  \mathcal{O}_{d}^{CR}\left(  M\right)  \right)  \leq \dim_{%
\mathbb{C}
}\left(  \mathcal{O}_{d}^{CR}\left(  \mathbb{H}^{n}\right)  \right)  .
\]
Furthermore, there exists a number $\delta \left(  \lambda \right)  >0$ such
that
\[
\dim_{%
\mathbb{C}
}\left(  \mathcal{O}_{\delta \left(  \lambda \right)  }^{CR}\left(  M\right)
\right)  =1.
\]

\end{corollary}

\end{document}